\documentclass[10pt]{amsart}

\theoremstyle{plain}
\newtheorem{theorem}{Theorem}
\newtheorem{proposition}{Proposition}
\newtheorem{corollary}{Corollary}
\newtheorem{lemma}{Lemma}

\theoremstyle{definition}
\newtheorem{definition}{Definition}

\newtheorem{remark}{Remark}

\newcommand{\rn}{\mathbb R}
\newcommand{\sn}{\mathbb S}

\newcommand{\la}{\langle}
\newcommand{\ra}{\rangle}

\DeclareMathOperator{\ricci}{Ricci}
\DeclareMathOperator{\grad}{grad}
\DeclareMathOperator{\riem}{Riem}

\DeclareMathOperator{\tr}{trace}
\DeclareMathOperator{\di}{div}
\DeclareMathOperator{\hess}{Hess}

\begin{document}

\title{Biharmonic surfaces of constant mean curvature}

\author{E. Loubeau}
\address{D{\'e}partement de Math{\'e}matiques \\
Universit{\'e} de Bretagne Occidentale \\
6, avenue Victor Le Gorgeu \\
CS 93837, 29238 Brest Cedex 3, France}
\email{Eric.Loubeau@univ-brest.fr}

\author{C. Oniciuc}
\address{Faculty of Mathematics, Al.I. Cuza University of Iasi, Bd. Carol I no. 11, 700506
Iasi, Romania}
\email{oniciucc@uaic.ro}

\thanks{C. Oniciuc was supported by a grant of the Romanian National Authority for Scientific Research, CNCS-UEFISCDI, project number PN-II-RU-TE-2011-3-0108}

\begin{abstract}
We compute a Simons' type formula for the stress-energy tensor of biharmonic maps from surfaces. Specializing to Riemannian immersions, we prove several rigidity results for biharmonic CMC surfaces, putting in evidence the influence of the Gaussian curvature on pseudo-umbilicity. Finally, the condition of biharmonicity is shown to enable an extension of the classical Hopf theorem to CMC surfaces in any ambient Riemannian manifold.
\end{abstract}

\maketitle

\section{Introduction}

While harmonic maps between abstract Riemannian manifolds are a generalization of minimal submanifolds, their study on two-dimensional domains remained nonetheless very valuable and brought new light to both theories. When, for topological or geometrical reasons, harmonic maps are non-existent or unsatisfactory, one can then measure the failure of harmonicity with the {\em bienergy functional}
$$E_2 (\phi) = \frac{1}{2} \int_{M} |\tau(\phi)|^2 \, v_g ,$$
where $M$ is compact, $\phi : (M,g) \to (N,h)$ is a smooth map and $\tau(\phi) = \tr \nabla d\phi$ is the tension field. 
Usual arguments (cf.~\cite{Jiang1}) show that critical points of $E_2$, called {\em biharmonic maps}, are solutions of 
$$\tau_{2} (\phi) = \tr (\nabla^{\phi})^2 \tau(\phi) - \tr R^{N} (d\phi(.),\tau(\phi)) d\phi(.) =0 ,$$
and we will use the adjective {\em proper} to designate non-harmonic biharmonic maps.

Whilst the interconnections between harmonic maps and minimal surfaces are clear and well established, in many cases, but not always, biharmonic Riemannian immersions have constant mean curvature (CMC). However, this articulation is not as clear as harmonicity and minimality, and the principal objective of this article is to explain how biharmonicity constrains CMC surfaces in an abstract ambient manifold. This is particularly well illustrated on compact biharmonic CMC surfaces whose Gaussian curvature has constant sign. They must be flat or pseudo-umbilical, if $K^M$ is non-negative (Corollary~\ref{cor3}), otherwise have pseudo-umbilical points (Theorem~\ref{thm2}). The role of pseudo-umbilical points in relaxing curvature constraints is further felt in the non-compact case, as their absence forces the CMC surface to be conformally flat 
(Theorem~\ref{thm2}). 

For complete surfaces, non-negative Gaussian curvature and an upper bound on the sectional curvature of the ambient space will cause the surface to be flat or pseudo-umbilical, but note that both can occur simultaneously (Proposition~\ref{prop3}). When the ambient manifold is a three-dimensional space form, the surface must be umbilical (Corollary~\ref{cor4}), confer~\cite{MO} for the classification.

Our approach is to derive, in Proposition~\ref{prop1}, a Simons' type formula for the biharmonic stress-energy tensor, valid for all biharmonic maps. As cumbersome as this equation is in the general case, on surfaces it simplifies enough (Proposition~\ref{prop2}) to enable the use of a divergence argument (Theorem~\ref{thm1}) and draw some consequences for biharmonic maps on a two-dimensional domain (Corollary~\ref{cor1} and~\ref{cor2}). However, the main consequences are for CMC surfaces.

To close the article, we show that, in any ambient space, the condition of biharmonicity preserves the holomorphicity of the Hopf differential of CMC surfaces (Theorem~\ref{thm3}).

Biharmonic CMC surfaces were also studied in~\cite{Fetcu,OZ} and \cite{S}.

The conventions we adopt are that the Riemann curvature tensor is
$$R(X,Y) = [\nabla_X, \nabla_Y ] - \nabla_{[X,Y]} ,$$
while its $(0,4)$ counterpart is
$$R(X,Y,Z,W) = \la R(X,Y)W,Z\ra .$$
The choice of sign for the Laplacians on sections and functions is the same and for functions on the real line $\Delta f = - f''$.

All objects, unless specified, are smooth and we assume summation on repeated indices, when apt.

\section{The biharmonic stress-energy tensor on surfaces}

Since biharmonic maps stem from a variational problem, one can apply the general principle of studying the same functional but under variations of the domain metric. This idea taken up on the bienergy leads to the biharmonic stress-energy tensor, which is symmetric and of type $(0,2)$ (\cite{LMO}).
\begin{definition}
Let $(M,g)$ and $(N,h)$ be Riemannian manifolds and $\phi : M \to N$ a smooth map.
The biharmonic stress-energy tensor of $\phi$ is
$$S_2(X,Y) = \left\{ \frac{|\tau(\phi)|^2}{2} + \la d\phi , \nabla\tau(\phi)\ra \right\} g(X,Y) - T(X,Y) ,$$
where
$$T(X,Y) = \la d\phi (X) , \nabla_{Y}\tau(\phi)\ra + \la d\phi(Y) , \nabla_{X}\tau(\phi)\ra .$$
\end{definition}

The main feature of $S_2$ is to satisfy Hilbert's principle of being divergence free at critical points i.e. (\cite{LMO,Jiang2}):
$$ \di S_2 = \la d\phi, \tau_2(\phi)\ra.$$

In order to exploit the biharmonicity of the map $\phi$, we compute the rough Laplacian of its biharmonic stress-energy tensor. This second order operator on $(0,2)$-tensors will reveal curvature terms which combine with the bitension field and formulas will involve swapping vector positions in the third fundamental form of $\phi$, with curvature appearing accordingly to a lemma we quote separately, without proof.

\begin{lemma}
\label{eq:thirdfundamentalform}
Let $\phi:(M^{m},g)\to (N^{n},h)$ be a smooth map. Then
\begin{eqnarray*}
(\nabla^{2}d\phi)(X,Y,Z) - (\nabla^{2}d\phi)(Z,Y,X)=R(X,Z)d\phi(Y) -d\phi(R^{M}(X,Z)Y),
\end{eqnarray*}
for any $X,Y,Z\in C(TM)$.
\end{lemma}

\begin{proposition}[The rough Laplacian of $S_2$] \label{prop1}
Let $(M,g)$ and $(N,h)$ be Riemannian manifolds and $\phi : M \to N$ a biharmonic map, then the (rough) Laplacian of $S_2$ is the symmetric $(0,2)$-tensor
\begin{align*}
&(\Delta^{R} S_2)(X,Y) = \Big( 2\la \Delta \tau(\phi), \tau(\phi)\ra - 2 |\nabla \tau(\phi)|^2 -2 \sum \la R(X_i,X_j)d\phi(X_i), \nabla_{X_j}\tau(\phi)\ra \\
& -2 \la d\phi(\ricci^{M}(.)),\nabla_{.}\tau(\phi)\ra -2\la \nabla d\phi , \nabla^{2} \tau(\phi) \ra + 2 \la d\phi(.) , \nabla_{.}(\Delta \tau(\phi))\ra  \Big) g(X,Y)\\
& + 2 \la \nabla_{X}\tau(\phi) , \nabla_{Y}\tau(\phi)\ra + \sum \la R(X_i,X)d\phi(X_i) , \nabla_{Y}\tau(\phi)\ra +\sum  \la R(X_i,Y)d\phi(X_i), \nabla_{X} \tau(\phi)\ra \\
&+ \la d\phi(\ricci^{M}(X)), \nabla_{Y}\tau(\phi)\ra + \la d\phi(\ricci^{M}(Y)), \nabla_{X}\tau(\phi)\ra \\
& + 2\sum  \la \nabla d\phi(X_i,X),(\nabla^2 \tau(\phi))(X_i,Y)\ra + 2\sum  \la \nabla d\phi(X_i,Y),(\nabla^2 \tau(\phi))(X_i,X)\ra \\
& -\la d\phi(X),\nabla_{Y}(\Delta (\tau(\phi)))\ra - \la d\phi(Y),\nabla_{X}(\Delta (\tau(\phi)))\ra 
+ \sum \la d\phi(X), R(X_i , Y)\nabla_{X_i}\tau(\phi)\ra \\
&+\sum  \la d\phi(Y), R(X_i , X)\nabla_{X_i}\tau(\phi)\ra 
+ \sum \la d\phi(X), \nabla_{X_i} R(X_i ,Y)\tau(\phi)\ra \\
&+\sum  \la d\phi(Y), \nabla_{X_i} R(X_i ,X)\tau(\phi)\ra 
+ \la d\phi(X), \nabla_{\ricci^{M}(Y)}\tau(\phi)\ra + \la d\phi(Y), \nabla_{\ricci^{M}(X)}\tau(\phi)\ra ,
\end{align*}
where $\{X_i\}$ is a geodesic frame around the point $p\in M$.
\end{proposition}

\begin{proof}
Let $\phi : (M,g) \to (N,h)$ a biharmonic map between Riemannian manifolds.
We will work with a geodesic frame $\{X_i\}$ around the point $p\in M$ and evaluate at the point $p$.

Writing out the Laplacian in the geodesic frame, yields
\begin{align*}
&\Delta(\la d\phi , \nabla\tau(\phi)\ra) \\
&= -\sum \Big\{\la \nabla_{X_i}[(\nabla d\phi)(X_i,X_j)],\nabla_{X_j}\tau(\phi)\ra + 2\la (\nabla d\phi)(X_i,X_j), (\nabla^{2}\tau(\phi))(X_i,X_j)\ra\\
&+ \la d\phi(X_j) , \nabla_{X_i}\nabla_{X_i}\nabla_{X_j} \tau(\phi)\ra - \la d\phi(X_j),\nabla_{X_i}\nabla_{\nabla_{X_i}{X_j}} \tau(\phi)\ra\Big\} ,
\end{align*}
and by the symmetry formula of the third fundamental form, we have
\begin{align*}
\sum \nabla_{X_i}[(\nabla d\phi)(X_i,X_j)] 
&= \nabla_{X_j}\tau(\phi) + \sum R(X_i,X_j)d\phi(X_i) - d\phi(\ricci^{M}(X_j))
\end{align*}
and
\begin{align*}
&\sum  \Big(\nabla_{X_i}\nabla_{X_i}\nabla_{X_j} \tau(\phi) - \nabla_{X_i}\nabla_{\nabla_{X_i}{X_j}} \tau(\phi) \Big) \\
&= \sum  \Big\{\nabla_{X_j}\nabla_{X_i}\nabla_{X_i} \tau(\phi) + \nabla_{[X_i,X_j]}\nabla_{X_i}\tau(\phi) + R(X_i,X_j)\nabla_{X_i}\tau(\phi) + \nabla_{X_i}R(X_i,X_j)\tau(\phi)\\
&-\Big( \nabla_{\nabla_{X_j}X_i}\nabla_{X_i}\tau(\phi) + \nabla_{[X_i,\nabla_{X_j}X_i]}\tau(\phi) + R(X_i,\nabla_{X_j}X_i)\tau(\phi)\Big)\Big\}\\
&= -\nabla_{X_j}\Big(\Delta\tau(\phi)\Big) + \sum  \Big\{(\nabla^{2}\tau(\phi))(X_j,\nabla_{X_i}X_i)+ R(X_i,X_j)\nabla_{X_i}\tau(\phi) \\
&+ \nabla_{X_i}R(X_i,X_j)\tau(\phi)\Big\} + \nabla_{\ricci^{M}(X_j)}\tau(\phi) ,
\end{align*}
since
\begin{align*}
\sum  [X_i,\nabla_{X_j}X_i] &= \sum  \nabla_{X_j}\nabla_{X_i}X_i + \ricci^{M}(X_j) .
\end{align*}
Therefore
\begin{align*}
&\Delta\Big( \la d\phi,\nabla\tau(\phi)\ra\Big) = -\sum  \la \nabla_{X_j}\tau(\phi),\nabla_{X_j}\tau(\phi)\ra - \sum \la R(X_i,X_j)d\phi(X_i), \nabla_{X_j}\tau(\phi)\ra \\
&- \sum  \la d\phi(\ricci^{M}(X_j)), \nabla_{X_j}\tau(\phi)\ra - 2 \la \nabla d\phi, \nabla^{2}\tau(\phi)\ra +\sum   \la d\phi(X_j),\nabla_{X_j}\Delta\tau(\phi)\ra \\
&-\sum \Big(\la d\phi(X_j),R(X_i,X_j)\nabla_{X_i}\tau(\phi)\ra + \la d\phi(X_j),\nabla_{X_i}R(X_i,X_j)\tau(\phi)\ra \\
&+ \la d\phi(X_j),\nabla_{\ricci^{M}(X_j)}\tau(\phi)\ra\Big) ,
\end{align*}
and as $\phi$ is biharmonic
\begin{align*}
\sum \la d\phi(X_j) ,\nabla_{X_i}R(&X_i,X_j)\tau(\phi)\ra \\
&= \sum X_i ( R^{N}(d\phi(X_i),d\phi(X_j),d\phi(X_j),\tau(\phi))) - \la \nabla d\phi ,R(.,.)\tau(\phi)\ra \\
&= \sum X_i \la R^{N}(d\phi(X_j),\tau(\phi))d\phi(X_j),d\phi(X_i)\ra\\
&= - \sum  \la \nabla_{X_i}\Delta \tau(\phi) , d\phi(X_i)\ra - \la \Delta \tau(\phi) , \tau(\phi)\ra ,
\end{align*}
while
\begin{align*}
\la\Delta\tau(\phi),\tau(\phi)\ra &= \sum   R^{N}(d\phi(X_i),\tau(\phi),d\phi(X_i),\tau(\phi)) ,
\end{align*}
\begin{align*}
\la R(X_i,X_j)d\phi(X_i), \nabla_{X_j}\tau(\phi)\ra &= \la R(X_i,X_j) \nabla_{X_i}\tau(\phi),d\phi(X_j)\ra ,
\end{align*}
and
\begin{align*}
\sum \la d\phi(\ricci^{M}(.)),\nabla\tau(\phi)\ra &= \sum \la d\phi(.),\nabla_{\ricci^{M}(.)}\tau(\phi)\ra ,
\end{align*}
so the Laplacian of the scalar term is
\begin{align*}
&\Delta\Big( \frac{|\tau(\phi)|^2}{2} + \la d\phi , \nabla\tau(\phi)\ra\Big) \\
&= 2 \la\Delta\tau(\phi),\tau(\phi)\ra - 2 |\nabla\tau(\phi)|^2 
-2 \sum  \la R(X_i,X_j)d\phi(X_i), \nabla_{X_j}\tau(\phi)\ra \\
&-2 \la d\phi(\ricci^{M}(.)),\nabla_{.}\tau(\phi)\ra 
-2 \la \nabla d\phi, \nabla^{2}\tau(\phi)\ra + 2 \la  d\phi(.), \nabla_{.}(\Delta\tau(\phi))\ra .
\end{align*}

On the other hand, to compute the (rough) Laplacian of the symmetric two-tensor 
\begin{align*}
T(X,Y) &= \la d\phi(X),\nabla_{Y}\tau(\phi)\ra + \la d\phi(Y),\nabla_{X}\tau(\phi)\ra ,
\end{align*}
we put $X=X_k$ and $Y=X_j$ and obtain, still evaluating expressions at the point $p$
\begin{align*}
&-(\Delta^{R}T)(X,Y) =\sum  \Big( \la \nabla_{X_i}\nabla_{X_i}d\phi(X),\nabla_{Y}\tau(\phi)\ra + 2\la \nabla_{X_i}d\phi(X),\nabla_{X_i}\nabla_{Y}\tau(\phi)\ra \\
&+ \la d\phi(X),\nabla_{X_i}\nabla_{X_i}\nabla_{Y}\tau(\phi)\ra + \la \nabla_{X_i}\nabla_{X_i}d\phi(Y),\nabla_{X}\tau(\phi)\ra 
+ 2\la \nabla_{X_i}d\phi(Y),\nabla_{X_i}\nabla_{X}\tau(\phi)\ra \\
&+ \la d\phi(Y),\nabla_{X_i}\nabla_{X_i}\nabla_{X}\tau(\phi)\ra-\la d\phi(Y),\nabla_{X_i}\nabla_{\nabla_{X_i}X}\tau(\phi)\ra - \la d\phi(X),\nabla_{X_i}\nabla_{\nabla_{X_i}Y}\tau(\phi)\ra \Big),
\end{align*}
since $\nabla_{X_i}\nabla_{X_i} X_j$ vanishes at the point $p$. This last expression simplifies further if we use the symmetries properties of the third fundamental form of $\phi$ to obtain
\begin{align*}
\sum \nabla_{X_i}\nabla_{X_i}d\phi(X) &=\sum (\nabla^{2} d\phi)(X_i,X_i,X)\\
&= \nabla_{X}\tau(\phi) + \sum R(X_i ,X)d\phi(X_i) - d\phi(\ricci^{M}(X)) ,
\end{align*}
and the curvature tensor of the pull-back bundle for
\begin{align*}
\sum \Big( \la d\phi(X),\nabla_{X_i}\nabla_{X_i}\nabla_{Y}&\tau(\phi)\ra - \la d\phi(X),\nabla_{X_i}\nabla_{\nabla_{X_i}Y}\tau(\phi)\ra \Big) \\
=&-\la d\phi(X), \nabla_{Y}(\Delta\tau(\phi)) \ra + \sum \la d\phi(X), R(X_i ,Y)\nabla_{X_i}\tau(\phi)\ra \\
&+ \sum \la d\phi(X),\nabla_{X_i}R(X_i ,Y)\tau(\phi)\ra + \la d\phi(X),\nabla_{\ricci^{M}(Y)}\tau(\phi)\ra .
\end{align*}
The Laplacian of the tensor $T$ is then equal to
\begin{align*}
-(\Delta T)(X,Y) &= \la \nabla_{X}\tau(\phi),\nabla_{Y}\tau(\phi)\ra + \sum \la R(X_i ,X)d\phi(X_i),\nabla_{Y}\tau(\phi)\ra \\
&+ \la d\phi(\ricci^{M}(X)),\nabla_{Y}\tau(\phi)\ra 
+ 2 \sum \la \nabla d\phi(X_i,X), (\nabla^{2}\tau(\phi))(X_i,Y)\ra  \\
&- \la d\phi(X), \nabla_{Y}(\Delta\tau(\phi)) \ra+ \sum \la d\phi(X), R(X_i ,Y)\nabla_{X_i}\tau(\phi)\ra \\
&+ \sum \la d\phi(X),\nabla_{X_i}R(X_i ,Y)\tau(\phi)\ra + \la d\phi(X),\nabla_{\ricci^{M}(Y)}\tau(\phi)\ra\\
&+\la \nabla_{Y}\tau(\phi),\nabla_{X}\tau(\phi)\ra + \sum \la R(X_i ,Y)d\phi(X_i),\nabla_{X}\tau(\phi)\ra \\
&+ \la d\phi(\ricci^{M}(Y)),\nabla_{X}\tau(\phi)\ra + 2 \sum \la \nabla d\phi(X_i,Y), (\nabla^{2}\tau(\phi))(X_i,X)\ra \\
&- \la d\phi(Y), \nabla_{X}(\Delta\tau(\phi)) \ra + \sum \la d\phi(Y), R(X_i ,X)\nabla_{X_i}\tau(\phi)\ra \\
&+ \sum \la d\phi(Y),\nabla_{X_i}R(X_i ,X)\tau(\phi)\ra 
+ \la d\phi(Y),\nabla_{\ricci^{M}(X)}\tau(\phi)\ra ,
\end{align*}
but
\begin{align*}
&\sum \la \nabla_{X_i}R(X_i,Y)\tau(\phi) , d\phi(X)\ra = \sum \la \nabla_{X_i}R^{N}(d\phi(X),\tau(\phi))d\phi(Y),d\phi(X_i)\ra \\
&- R^{N}(d\phi(X),\tau(\phi),d\phi(Y),\tau(\phi)) - \sum \la R(X_i,Y)\tau(\phi) ,(\nabla d\phi)(X_i,X)\ra ,
\end{align*}
and, since $\phi$ is biharmonic,
\begin{align*}
&-\la d\phi(X),\nabla_{Y}(\Delta\tau(\phi))\ra = \sum \la \nabla_{Y} (R^{N}(d\phi(X_i),d\phi(X))d\phi(X_i)) ,\tau(\phi)\ra \\
&+ \sum \Big(\la R^{N}(d\phi(X_i),d\phi(X))d\phi(X_i) ,\nabla_{Y}\tau(\phi)\ra - \la (\nabla d\phi)(X,Y),R(X_i ,\tau(\phi))X_i \ra \Big).
\end{align*}
Summing the various parts together yields the proposition.
\end{proof}

While the general expression for the rough Laplacian of $S_2$, at first, seems unwieldy, in a manner reminiscent of its harmonic counterpart (cf.~\cite{BLO}), it becomes amenable when the domain is a surface. The final formula only involves three ingredients: the tensor $S_2$ itself, the Gaussian curvature and the norm of the tension field of the map. This paves the way for a series of propositions and corollaries, for both maps and Riemannian immersions, where topological and curvature conditions restrict the existence of biharmonic maps.
\begin{proposition}\label{prop2}
Let $\phi : (M^2,g) \to (N,h)$ be a biharmonic map defined on a surface $M^2$. The Laplacian of its biharmonic stress-energy tensor is
$$
\Delta^{R}S_2= - 2 K^{M}S_2 + \nabla d (|\tau(\phi)|^2) +\Big\{K^{M}|\tau(\phi)|^2 + \Delta|\tau(\phi)|^2\Big\}g ,
$$
where $K^M$ is the Gaussian curvature of $(M^2,g)$.
\end{proposition}
\begin{proof}
Since $\dim M=2$, its Ricci curvature is $\ricci^{M}=K^{M} I$, with $K^{M}\in C^{\infty}(M)$. We will work with a geodesic frame $\{ X_1,X_2\}$ around a point $p\in M^2$ and evaluate all expressions at this point.

As $\Delta^{R} S_2$ is a symmetric $(0,2)$-tensor, there are only two cases to consider and from the previous proposition, combined with  basic symmetries of the curvature tensor and the biharmonicity condition, we have
\begin{align*}
&(\Delta^{R} S_2)(X_1,X_2) =2\la\nabla_{X_1} \tau(\phi),\nabla_{X_2} \tau(\phi)\ra - \la d\phi(X_2) , \nabla_{X_1}(\Delta \tau(\phi))\ra\\
&+ 2K^{M}\Big\{\la d\phi(X_1), \nabla_{X_2}\tau(\phi)\ra  + \la d\phi(X_2), \nabla_{X_1}\tau(\phi)\ra\Big\}\\
&+ 2\la \nabla d\phi(X_1,X_2) , - \Delta\tau(\phi)\ra + 2\la \nabla d\phi(X_1,X_1) , (\nabla^{2} \tau(\phi))(X_1,X_2) \ra \\
&+ 2\la \nabla d\phi(X_2,X_2) , (\nabla^{2} \tau(\phi))(X_2,X_1) \ra 
- \la d\phi(X_1) , \nabla_{X_2}(\Delta \tau(\phi))\ra  \\
& + \la d\phi(X_1),\nabla_{X_1}R(X_1 ,X_{2})\tau(\phi)\ra + \la d\phi(X_2),\nabla_{X_2}R(X_2 ,X_{1})\tau(\phi)\ra \\
&=2\la\nabla_{X_1} \tau(\phi),\nabla_{X_2} \tau(\phi)\ra + 2K^{M}\Big\{\la d\phi(X_1), \nabla_{X_2}\tau(\phi)\ra  + \la d\phi(X_2), \nabla_{X_1}\tau(\phi)\ra\Big\}\\
&+ 2\la \nabla d\phi(X_1,X_1) , (\nabla^{2} \tau(\phi))(X_1,X_2) \ra + 2\la \nabla d\phi(X_2,X_2) , (\nabla^{2} \tau(\phi))(X_2,X_1) \ra \\
&- \la \nabla_{X_1}d\phi(X_1),R(X_1 ,X_{2})\tau(\phi)\ra - \la \nabla_{X_2}d\phi(X_2),R(X_2 ,X_{1})\tau(\phi)\ra .
\end{align*}
But
\begin{align*}
& 2\la \nabla d\phi(X_1,X_1) , (\nabla^{2} \tau(\phi))(X_1,X_2) \ra - \la \nabla d\phi(X_1,X_1),R(X_1 ,X_{2})\tau(\phi)\ra \\
&=\la \nabla d\phi(X_1,X_1) , 2\nabla_{X_1}\nabla_{X_2} \tau(\phi) - \nabla_{X_1}\nabla_{X_2} \tau(\phi) + \nabla_{X_2}\nabla_{X_1} \tau(\phi) \ra ,
\end{align*}
so
\begin{align*}
&(\Delta^{R} S_2)(X_1,X_2) = 2K^{M}\Big\{\la d\phi(X_1), \nabla_{X_2}\tau(\phi)\ra  + \la d\phi(X_2), \nabla_{X_1}\tau(\phi)\ra\Big\}\\
& 2\la\nabla_{X_1} \tau(\phi),\nabla_{X_2} \tau(\phi)\ra  + \la \tau(\phi), \nabla_{X_1}\nabla_{X_2} \tau(\phi) + \nabla_{X_2}\nabla_{X_1} \tau(\phi) \ra .
\end{align*}
Since
\begin{align*}
(\nabla d|\tau(\phi)|^2)(X_1,X_2) &= \la \nabla_{X_1}\nabla_{X_2} \tau(\phi) + \nabla_{X_2}\nabla_{X_1} \tau(\phi),\tau(\phi) \ra +2 \la \nabla_{X_1} \tau(\phi),\nabla_{X_2}\tau(\phi) \ra ,
\end{align*}
we deduce that
\begin{align*}
&(\Delta^{R} S_2)(X_1,X_2) =-2 K^{M} S_2 (X_1,X_2) + (\nabla d|\tau(\phi)|^2)(X_1,X_2) .
\end{align*}
The other case to look at is when the two vectors are the same and then Proposition~\ref{prop1} shows that, using the symmetries of $R^N$, 
\begin{align*}
&(\Delta^{R} S_2)(X_1,X_1) =-2 \la R^{N}(X_1,\tau(\phi))d\phi(X_1),\tau(\phi)\ra -2 \la R^{N}(X_2,\tau(\phi))d\phi(X_2),\tau(\phi)\ra \\
&-2 \la \nabla_{X_2}\tau(\phi),\nabla_{X_2}\tau(\phi)\ra 
-2 K^{M}\la d\phi(X_2),\nabla_{X_2}\tau(\phi)\ra 
-2\la \nabla d\phi (X_2,X_2), (\nabla^{2} \tau(\phi))(X_2,X_2) \ra \\
&- 2\la \nabla d\phi (X_1,X_2), (\nabla^{2} \tau(\phi))(X_1,X_2) \ra
+2 \la d\phi(X_1) , \nabla_{X_1}(\Delta \tau(\phi))\ra + 2 \la d\phi(X_2) , \nabla_{X_2}(\Delta \tau(\phi))\ra \\
&+ 2K^{M}\la d\phi(X_1), \nabla_{X_1}\tau(\phi)\ra 
+2\la \nabla d\phi (X_1,X_1), (\nabla^{2} \tau(\phi))(X_1,X_1) \ra \\
&+ 2\la \nabla d\phi (X_2,X_1), (\nabla^{2} \tau(\phi))(X_2,X_1) \ra 
-2 \la d\phi(X_1) , \nabla_{X_1}(\Delta \tau(\phi))\ra \\
&+ 2 \la d\phi(X_1),\nabla_{X_2}R(X_2 ,X_{1})\tau(\phi)\ra \\
&=-2 |\nabla_{X_2}\tau(\phi)|^2 -2 K^{M}\la d\phi(X_2),\nabla_{X_2}\tau(\phi)\ra +2 K^{M}\la d\phi(X_1),\nabla_{X_1}\tau(\phi)\ra\\
&-2 \la R^{N}(d\phi(X_1),\tau(\phi))d\phi(X_1),\nabla d\phi(X_1,X_1)\ra -2 X_2 \la d\phi(X_2) , R^{N}(d\phi(X_1),\tau(\phi))d\phi(X_1) \ra \\
&-2 \la R^{N}(d\phi(X_2),\tau(\phi))d\phi(X_2),\nabla d\phi(X_1,X_1)\ra
-2\la \nabla d\phi (X_2,X_2), (\nabla^{2} \tau(\phi))(X_2,X_2) \ra \\
&+2\la \nabla d\phi (X_1,X_1), (\nabla^{2} \tau(\phi))(X_1,X_1) \ra
-2\la \nabla d\phi (X_1,X_2), R(X_1,X_2)\tau(\phi) \ra \\
&+ 2 \la d\phi(X_1),\nabla_{X_2}R(X_2 ,X_{1})\tau(\phi)\ra ,
\end{align*}
since:
\begin{align*}
i)&- 2\la \nabla d\phi (X_1,X_2), (\nabla^{2} \tau(\phi))(X_1,X_2) + 2 \la \nabla d\phi (X_2,X_1), (\nabla^{2} \tau(\phi))(X_2,X_1) \ra \\
&=- 2\la \nabla d\phi (X_1,X_2), R(X_1,X_2) \tau(\phi) \ra ;\\
ii)& -2 \la R^{N}(d\phi(X_1),\tau(\phi))d\phi(X_1),\tau(\phi)\ra -2 \la d\phi(X_2) , \nabla_{X_2}R^{N}(d\phi(X_1),\tau(\phi))d\phi(X_1) \ra \\
&=-2 \la R^{N}(d\phi(X_1),\tau(\phi))d\phi(X_1),\nabla_{X_1}d\phi(X_1)\ra -2 \la R^{N}(d\phi(X_1),\tau(\phi))d\phi(X_1),\nabla_{X_2}d\phi(X_2)\ra \\
&-2 X_2 \la d\phi(X_2) , R^{N}(d\phi(X_1),\tau(\phi))d\phi(X_1) \ra   +2 \la \nabla_{X_2}d\phi(X_2) , R^{N}(d\phi(X_1),\tau(\phi))d\phi(X_1) \ra ;\\
iii)& -2 \la R^{N}(d\phi(X_2),\tau(\phi))d\phi(X_2),\tau(\phi)\ra -2 \la d\phi(X_2) , \nabla_{X_2}R^{N}(d\phi(X_2),\tau(\phi))d\phi(X_2) \ra \\
&= -2 \la R^{N}(d\phi(X_2),\tau(\phi))d\phi(X_2),\nabla d\phi(X_1,X_1)\ra .
\end{align*}
Observe that
\begin{align*}
&- X_2 \la d\phi(X_2) , R^{N}(d\phi(X_1),\tau(\phi))d\phi(X_1) \ra + \la d\phi(X_1),\nabla_{X_2}R(X_2 ,X_{1})\tau(\phi)\ra \\
&= - X_2 R^{N}(d\phi(X_2),d\phi(X_1),d\phi(X_1),\tau(\phi)) + X_2 R^{N}(d\phi(X_2),d\phi(X_1),d\phi(X_1),\tau(\phi)) \\
&+ \la \nabla d\phi (X_1,X_2), R(X_1,X_2)\tau(\phi) \ra ,
\end{align*}
so
\begin{align*}
&(\Delta^{R} S_2)(X_1,X_1) =-2 |\nabla_{X_2}\tau(\phi)|^2 -2 K^{M}\la d\phi(X_2),\nabla_{X_2}\tau(\phi)\ra +2 K^{M}\la d\phi(X_1),\nabla_{X_1}\tau(\phi)\ra\\
&-2 \la R^{N}(d\phi(X_1),\tau(\phi))d\phi(X_1),\nabla d\phi(X_1,X_1)\ra +2\la \nabla d\phi (X_1,X_1), (\nabla^{2} \tau(\phi))(X_1,X_1) \ra\\
& -2 \la R^{N}(d\phi(X_2),\tau(\phi))d\phi(X_2),\nabla d\phi(X_1,X_1)\ra -2\la \nabla d\phi (X_2,X_2), (\nabla^{2} \tau(\phi))(X_2,X_2) \ra .
\end{align*}
But
\begin{align*}
\la \tau(\phi), \nabla_{X_2}\nabla_{X_2}\tau(\phi) \ra &= -\tfrac{1}{2}\Delta|\tau(\phi)|^2 - \tfrac{1}{2} X_1 X_1(|\tau(\phi)|^2) - |\nabla_{X_2}\tau(\phi)|^2 ,
\end{align*}
so
\begin{align*}
(\Delta^{R} S_2)(X_1,X_1) &=-2 K^{M}S_2(X_1,X_1) + \Big\{K^{M}|\tau(\phi)|^2 + \Delta|\tau(\phi)|^2\Big\}g(X_1,X_1) \\
&+ (\nabla d |\tau(\phi)|^2)( X_1, X_1) ,
\end{align*}
with a similar expression for $(\Delta^{R} S_2)(X_2,X_2)$.\\
Therefore
\begin{align*}
\Delta^{R}S_2 &= - 2 K^{M}S_2 + \nabla d (|\tau(\phi)|^2) +\Big\{K^{M}|\tau(\phi)|^2 + \Delta|\tau(\phi)|^2\Big\}g .
\end{align*}
\end{proof}

The expression for the Laplacian of the biharmonic stress-energy tensor on a surface is simple enough to be contracted with $S_2$ itself and combined with the divergence theorem, if the domain is assumed to be compact. The ensuing integral formula tightly binds the tensor $S_2$, the Gaussian curvature and the norm of the tension field together, and conditions on two of them determine the third.

More geometrical applications will be found for Riemannian immersions in the next section.

\begin{theorem} \label{thm1}
Let $\phi : M^2 \to N^n$ be a biharmonic map and assume $M^2$ is compact. Then
$$\int_{M} |\nabla S_2|^2 \, v_g + 2\int_{M} K^{M} \Big(|S_2|^2 -\frac{|\tau(\phi)|^4}{2}\Big) \, v_g = \int_{M} |d(|\tau(\phi)|^2)|^2 \, v_g ,$$
where $K^M$ is the Gaussian curvature of $(M^2,g)$.
\end{theorem}
\begin{proof}
Observe that
$$ \di \la S_2 , d(|\tau(\phi)|^2)\ra = \la \di S_2, d(|\tau(\phi)|^2)\ra + \la S_2 , \hess(|\tau(\phi)|^2)\ra .$$
As $\di S_2 =0$, we have
$$\int_{M} \la S_2 , \hess(|\tau(\phi)|^2)\ra \, v_g = \int_{M} \di \Big\{ \la S_2 , d(|\tau(\phi)|^2)\ra\Big\} \, v_g = 0 ,$$
which combined with the classical equality
$$ \int_M \la \Delta^R S_2 , S_2 \ra \, v_g = \int_{M} |\nabla S_2|^2 \, v_g ,$$
gives the theorem.
\end{proof}

\begin{remark} \label{remark1}
Note that the term $ 2 |S_2|^2 - |\tau(\phi)|^4 $ is always non-negative since equal to $\big((S_2 (X_1,X_1) - S_2 (X_2, X_2)\big)^2 + 4 S^2_2(X_1,X_2)$ and $|S_2|^2 = \frac{|\tau(\phi)|^4}{2}$ if and only if 
$S_2 = \frac{|\tau(\phi)|^2}{2} g$.
\end{remark}

A biharmonic map with parallel stress-energy tensor must have a tension field of constant norm~\cite{LMO}, but Proposition~\ref{prop2} shows greater restrictions for two dimensional domains.

\begin{corollary}\label{cor1}
Let $\phi : M^2 \to N^n$ be a biharmonic map, assume $M$ is compact and $\nabla S_2 =0$. Then $|\tau(\phi)|$ constant and $\int_{M}  K^{M} \, v_g = 0$ or $S_2 = \frac{|\tau(\phi)|^2}{2} g$.
\end{corollary}
\begin{proof}
If $\nabla S_2 =0$, then its norm and trace, $|\tau(\phi)|^2$, are constant, hence
$$\Big(|S_2|^2 - \frac{|\tau(\phi)|^4}{2}\Big) \int_{M} K^{M} \, v_g =0 .$$
\end{proof}

If the norm of the tension field is constant, we can deduce a partial converse for non-negative curvature.

\begin{corollary} \label{cor2}
Let $\phi : (M^2 ,g) \to (N^n ,h)$ be a proper-biharmonic map with $|\tau(\phi)|^2$ constant. Assume $M$ is compact and $K^{M}\geq 0$. Then $S_2$ is parallel and $M$ is flat or $S_2 = \frac{|\tau(\phi)|^2}{2} g$.
\end{corollary}

\section{Constant mean curvature surfaces}

To be able to offer conditions with greater geometrical content, we concentrate our applications on Riemannian immersions. The recurrent condition on the map is pseudo-umbilicity, as an equality between the shape operator $A_H$ in the direction of the mean curvature vector field $H$ and the metric.

The pivotal role of pseudo-umbilical immersions, already observed in the study of the biharmonic stress-energy tensor (cf.~\cite{LMO}), emerges again in link with the curvature of the domain surface, sometimes to the extent of determining its topology.

In the absence of compactness, the divergence theorem is substituted by a parabolicity argument on constant mean curvature surfaces, associated with a bound on the curvature tensor of the target space.

Finally, working with complex coordinates on a Riemann surface, the $(2,0)$-part of the $H$-component of the second fundamental form $B$ is shown to be holomorphic if and only if the mean curvature is constant.

Recall that if $\phi : M^2 \to N$ is a pseudo-umbilical proper-biharmonic Riemannian immersion then the norm of its mean curvature vector field is constant.
As a consequence, and since $S_2 = -2|H|^2 g + 4 A_H$, a re-wording of Corollaries~\ref{cor1} and~\ref{cor2} is

\begin{corollary}\label{cor2.5}
Let $\phi : (M^2 ,g) \to (N^n ,h)$ be a proper-biharmonic Riemannian immersion with  $\nabla A_H =0$, from a compact oriented surface. Then $M$ is topologically a torus or pseudo-umbilical.
\end{corollary}

\begin{corollary}\label{cor3}
Let $\phi : (M^2 ,g) \to (N^n ,h)$ be a proper-biharmonic Riemannian immersion with $|H|^2$ constant. Assume $M$ compact and $K^{M}\geq 0$. Then $\nabla A_H =0$ and $M$ is flat or pseudo-umbilical.
\end{corollary}

The next result shows that pseudo-umbilical points allow some flexibility of the curvature, since away from these points special coordinates exist, in which the metric is conformally flat (with a globally defined factor), the shape operator has a simple expression, while its eigenvalues can be computed from the mean curvature vector field.

\begin{theorem} \label{thm2}
Let $\phi : (M^2 ,g) \to (N^n ,h)$ be a proper-biharmonic Riemannian immersion with constant mean curvature. 
We denote by $\lambda_1$ and $\lambda_2$ the principal curvatures of $M$ corresponding to $A_H$, $\lambda_1 \geq \lambda_2$, and let $\mu = \lambda_1 - \lambda_2$. Consider $p\in M$ such that $\mu(p)>0$, i.e. $p$ is a non pseudo-umbilical point. Then, around $p$, there is a local chart $(U;x,y)$ which is both isothermal and a line of curvature coordinate system for $A_H$. We have, on $U$,
$$ g =\frac{1}{\mu} ( dx^2 + dy^2), \quad \la A_H (.) , .\ra = \frac{1}{\mu} ( \lambda_1 dx^2 + \lambda_2 dy^2) ,$$
$$\sum_{i=1}^2 R^{N}(X_i,H,X_i,H) - |\nabla^{\perp}H|^2 -2 |H|^4 >0 ,$$
and
\begin{align*}
\lambda_1 &= |H|^2 + \frac{\sqrt{2}}{2}\sqrt{\sum_{i=1}^2 R^{N}(X_i,H,X_i,H) - |\nabla^{\perp}H|^2 -2 |H|^4} ,\\
\lambda_2 &= |H|^2 - \frac{\sqrt{2}}{2}\sqrt{\sum_{i=1}^2 R^{N}(X_i,H,X_i,H) - |\nabla^{\perp}H|^2 -2 |H|^4} ,\\
\end{align*}
with $X_1 = \sqrt{\mu} \partial_x$, $X_2 = \sqrt{\mu} \partial_y$. Moreover, 
$$\Delta \ln \Big( \sum_{i=1}^2 R^{N}(X_i,H,X_i,H) - |\nabla^{\perp}H|^2 -2 |H|^4 \Big) = -4 K^M ,$$
and the Gauss equation becomes
$$\riem^{N} (X_1,X_2) = K^M - 2 |H|^2 + \frac{1}{2 |H|^2} \sum_{i=1}^2 R^{N}(X_i,H,X_i,H) .$$
\end{theorem}
\begin{proof}
Let $\lambda_1$ and $\lambda_2$ be the principal curvatures in the direction of $H$, i.e. $\lambda_1$ and $\lambda_2$ are the eigenvalues of $A_{H}$. In an open neighbourhood $U$ around a non pseudo-umbilical point $p$, $\lambda_1 >\lambda_2$ on $U$, $\lambda_1 ,\lambda_2 \in C^{\infty}(U)$ (in general they are only continuous) and therefore $\mu=\lambda_1 - \lambda_2$ is a positive smooth function on $U$.

Let $\{X_1 ,X_2\}$ be a local orthonormal frame on $U$ such that $A_{H}(X_1)=\lambda_1 X_1$ and $A_{H}(X_2)=\lambda_2 X_2$. We consider $\omega^{2}_{1}$ and $\omega_{2}^{1} \in \bigwedge^{1} (U)$ defined by
$$\nabla X_1 = \omega^{2}_{1} X_2 \mbox{ and } \nabla X_2 = \omega^{1}_{2} X_1 .$$
Clearly $\omega^{2}_{1}=-\omega^{1}_{2}$.\\
If we put $X=Z=X_1$ and $Y=X_2$, the Codazzi equation becomes
\begin{align*}
R^{N}(X_1,H,X_2,X_1) &=-\omega^{1}_{2}(X_1)\mu - X_2 \lambda_1 - \la B(X_2 ,X_1),\nabla_{X_1}^{\perp}H \ra\\
&+ \la B(X_1 ,X_1),\nabla_{X_2}^{\perp}H\ra .
\end{align*}
Recall that the tangent part of the biharmonic equation is
$$\la B(X_2 ,X_1),\nabla_{X_1}^{\perp}H \ra + \la B(X_2 ,X_2),\nabla_{X_2}^{\perp}H \ra + R^{N}(X_1,H,X_2,X_1) =0 ,$$
thus
\begin{align*}
\omega^{1}_{2}(X_1)\mu + X_2 \lambda_1 = 2 \la H, \nabla_{X_2}^{\perp}H \ra =0 ,
\end{align*}
and
$$ \omega^{1}_{2}(X_1) =- \frac{X_2 \lambda_1}{\mu} .$$
Note that
\begin{align*}
X_2 (\lambda_2) &= X_2 \la A_{H}(X_2),X_2\ra = X_2 \la B(X_2,X_2),H\ra = -X_2 \lambda_1 ,
\end{align*}
therefore 
$$ \omega^{1}_{2}(X_1) = \tfrac{1}{2}\Big( - \frac{X_2 \lambda_1}{\mu} + \frac{X_2 \lambda_2}{\mu}\Big)= -\tfrac{1}{2} \frac{X_2 \mu}{\mu} .$$
Exchanging $X_1$ and $X_2$, we similarly obtain
$$ \omega^{1}_{2}(X_2) = \tfrac{1}{2} \frac{X_1 \mu}{\mu} ,$$
therefore
$$\omega^{1}_{2} = -\tfrac{1}{2} \frac{X_2 \mu}{\mu}\omega_1 + \tfrac{1}{2} \frac{X_1 \mu}{\mu} \omega_2 .$$
The Gauss equation implies that
$$ d\omega^{1}_{2} (X_1,X_2)=K^{M} ,$$
i.e.
\begin{align*}
K^{M} &= \tfrac{1}{2} \Big( X_1 X_1 \ln \mu + X_2 X_2 \ln \mu \Big) -(\omega^{1}_{2}(X_1))^2 -(\omega^{1}_{2}(X_2))^2 ,
\end{align*}
but
\begin{align*}
\nabla_{X_1}X_1 &= \tfrac{1}{2} \Big( X_2 \ln \mu\Big) X_2  \, ; \quad
\Big(\nabla_{X_1}X_1\Big)( \ln \mu) = \tfrac{1}{2} \Big( X_2 \ln \mu\Big)^2 ;\\
(\omega^{1}_{2}(X_1))^2 &= \tfrac{1}{4} \Big( X_2 \ln \mu\Big)^2 = \tfrac{1}{2}\Big(\nabla_{X_1}X_1\Big)( \ln \mu) ,
\end{align*}
while
\begin{align*}
\nabla_{X_2}X_2 &= \tfrac{1}{2} \Big( X_1 \ln \mu\Big) X_1 \, ; \quad
\Big(\nabla_{X_2}X_2\Big)( \ln \mu) = \tfrac{1}{2} \Big( X_1 \ln \mu\Big)^2 ;\\
(\omega^{1}_{2}(X_2))^2 &= \tfrac{1}{4} \Big( X_1 \ln \mu\Big)^2 = \tfrac{1}{2}\Big(\nabla_{X_2}X_2\Big)( \ln \mu) .
\end{align*}
Therefore
$$ \Delta \ln\mu = -2 K^{M}.$$
Since 
$$ \left[ \tfrac{1}{\sqrt{\mu}} X_1 , \tfrac{1}{\sqrt{\mu}} X_2 \right] =0,$$
there exist coordinate functions $(x,y)$ on $U$, such that $\frac{\partial}{\partial x} = \frac{1}{\sqrt{\mu}} X_1 $ and $\frac{\partial}{\partial y} = \frac{1}{\sqrt{\mu}} X_2 $. Moreover, the normal part of the biharmonicity equation:
$$ \Delta^{\perp} H + \tr B(.,A_H .) + \tr (R^N (.,H). )^\perp =0 , $$
becomes, when $H$ is constant,
$$ |\nabla^\perp H|^2 + |A_H|^2 - \sum_{i=1}^2 R^N (X_i ,H, X_i, H) =0$$
and, since
\begin{align*}
\lambda_1 + \lambda_2 &= 2 |H|^2  \mbox{ and } \, \lambda^2_1 + \lambda^2_2 =  |A_H|^2 ,
\end{align*}
we deduce that
$$|A_H|^2 - 2|H|^4 = \frac{(\lambda_1 - \lambda_2)^2}{2},$$
hence
$$\lambda_1 - \lambda_2 = \sqrt{2}\sqrt{ \sum_{i=1}^2 R^N (X_i ,H, X_i, H) -|\nabla^\perp H|^2 -2 |H|^4 }.$$
\end{proof}

\begin{remark}
i) If $n=3$, we can replace $\sum_{i=1}^2 R^N (X_i ,H, X_i, H)$ by $\ricci^{N}(H,H)$.\\
ii) Let $\phi : (M^2 ,g) \to (N,h)$ be a proper-biharmonic Riemannian immersion with constant mean curvature. If $(M^2 ,g)$ is complete and has no pseudo-umbilical point then its universal cover is (globally) conformally equivalent to $\rn^2$.
\end{remark}

\begin{corollary}\label{cor4}
Let $\phi : (M^2 ,g) \to N^3 (c)$ be a proper-biharmonic Riemannian immersion with constant mean curvature in a three dimensional real space form. Then it is umbilical.
\end{corollary}
\begin{proof}
If there exists a non-umbilical point $p_0 \in M$, then, around $p_0$, we have
$$
\riem^N (X_1,X_2) = K^M -2 |H|^2- \frac{1}{2|H|^2}\ricci^N (H,H)
$$
and
$$
K^M = -\frac{1}{4} \Delta \ln \Big( \ricci^N (H,H) - 2 |H|^4\Big) ,
$$
but $\ricci^N (H,H) = 2c |H|^2$ is constant, so $K^M$ is zero. On the other hand, the first equation implies that
$ c= K^M - 2 |H|^2 + c$, which contradicts $K^M=0$.
\end{proof}

As the formulas for $\lambda_1$ and $\lambda_2$ in Theorem~\ref{thm2} remain valid also for pseudo-umbilical points, we deduce.
\begin{corollary}
Let $\phi : (M^2 ,g) \to (N^3,h)$ be a proper-biharmonic Riemannian immersion with constant mean curvature. Assume that there exists $c>0$ such that $\ricci^N (U,U) \geq c |U|^2$ with $|H|^2 \in (0,\frac{c}{2})$. Then $M^2$ has no pseudo-umbilical point.
\end{corollary}

\begin{corollary}
Let $\phi : (M^2 ,g) \to (N^n ,h)$ be a proper-biharmonic Riemannian immersion with constant mean curvature. Assume $M$ is compact, oriented and has no pseudo-umbilical point, then $M$ is topologically a torus.
\end{corollary}

\begin{corollary}
Let $\phi : (M^2 ,g) \to (N^n ,h)$ be a proper-biharmonic Riemannian immersion. Assume that $\lambda_1$ and $\lambda_2$ are constant, then $\nabla A_H =0$ and $M$ is flat or pseudo-umbilical.
\end{corollary}

If $M$ is not compact, we need some assumption on the curvature of the target space (cf. also \cite[Prop. 4.6 and 4.7]{Fetcu}).
\begin{proposition}\label{prop3}
Let $\phi : (M^2 ,g) \to (N^n ,h)$ be a proper-biharmonic Riemannian immersion with constant mean curvature. Assume $M$ is non-compact, complete and $K^{M}$ is non-negative. Assume that $\riem^{N}\leq K_0$ where $K_0>0$ (in the sense that $R^{N}(U,V,U,V)\leq K_0$ for all $\{U,V\}$ orthonormal). Then $\nabla A_H =0$ and $M$ is flat or pseudo-umbilical.
\end{proposition}
\begin{proof}
By the previous formulas for the Laplacian of $S_2$, we have
\begin{align*}
-\tfrac{1}{2} \Delta |S_2|^2 &= -\la \Delta^{R}S_2 , S_2\ra + |\nabla S_2|^2\\
&=  K^{M}\left(2|S_2|^2 -|\tau(\phi)|^4\right) + |\nabla S_2|^2 ,
\end{align*}
which must be non-negative (Remark~\ref{remark1}), therefore $|S_2|^2$ is a subharmonic function and bounded from above since, for Riemannian immersions, $|S_2|^2 = 8( 2|A_H|^2 -3 |H|^4)$ and $|A_H|^2$ is itself bounded from above.
Indeed, if $\phi$ is biharmonic then
$$\Delta^{\perp}H + \tr B(.,A_{H}.) + \tr (R^{N}(.,H).)^{\perp} =0 ,$$
and
\begin{align*}
&\la\Delta^{\perp}H, H\ra =- |A_{H}|^2 + \sum_{i=1}^{2} R^{N}(X_i,H,X_i,H) \\
\end{align*}
but as $|H|$ is constant $\la\Delta^{\perp}H, H\ra = |\nabla^{\perp} H|^2$,
therefore
\begin{align*}
|A_{H}|^2 &= -|\nabla^{\perp} H|^2 + \sum_{i=1}^{2} R^{N}(X_i,H,X_i,H) \\
&\leq \sum_{i=1}^{2} R^{N}(X_i,H,X_i,H) \leq 2|H|^2 K_0
\end{align*}
and $|A_{H}|^2 \leq 2 K_0 |H|^2$.
As $M$ is complete with $K^M$ non-negative, it is parabolic and $|S_2|^2$, a subharmonic function bounded from above, must be constant:
$$
K^{M}\left(|A_{H}|^2 -4|H|^4\right) = 0 ,
$$
while $\nabla A_{H} =0$, in particular, $|A_{H}|^2$ is constant.
\end{proof}

\begin{remark}
When the dimension of the target is three, we can replace the curvature condition by a upper bound on the Ricci tensor.
\end{remark}

The Hopf Theorem~\cite{Hopf} shows that a compact simply-connected surface of constant mean curvature immersed in a three-dimensional Euclidean space must be umbilical, hence an embedded round sphere, and the condition of biharmonicity allows to extend this to any codomain. This result has some strict implications on the set of pseudo-umbilical points and hints at the difficulties of working with non constant mean curvature surfaces. An interesting parallel has to be drawn with \cite{Fetcu}.

 \begin{theorem} \label{thm3}
Let $\phi : (M^2 ,g) \to (N^n ,h)$ be a proper-biharmonic Riemannian immersion with mean curvature vector field $H$, $M^2$ oriented. Let $z$ be a complex coordinate on $M^2$ then the function $\la B(\partial z,\partial z),H\ra$ is holomorphic if and only if the norm of $H$ is constant.
\end{theorem}
\begin{proof}
Let $\phi : (M^2 ,g) \to (N^n ,h)$ be a proper-biharmonic Riemannian immersion with mean curvature $H$. Then the tangent part of the biharmonic equation is:
$$ \grad \frac{|H|^2}{2} + \tr A_{\nabla^{\perp}_{.}H}(.) + \tr (R^{N}(d\phi(.),H)d\phi(.))^{T} =0 .$$
Assume $M^2$ is orientable then $M^2$ is a one-dimensional complex manifold. Let $g= \lambda^2 \Big( dx^2 + dy^2\Big)$ and
\begin{align*}
&\frac{1}{2} \partial_x (|H|^2) \partial_x+ \frac{1}{2} \partial_y (|H|^2)\partial_x +A_{\nabla^{\perp}_{\partial x}H}(\partial x) + A_{\nabla^{\perp}_{\partial y}H}(\partial y)\\
&+ (R^{N}(\partial x,H)\partial x+ R^{N}(\partial y,H)\partial y)^{T} =0 ,
\end{align*}
therefore
\begin{align*}
& \frac{\lambda^2}{2} \partial_x (|H|^2) + \la A_{\nabla^{\perp}_{\partial x}H}(\partial x),\partial x\ra + \la A_{\nabla^{\perp}_{\partial y}H}(\partial y),\partial x\ra
+ R^{N}(\partial y,H,\partial x,\partial y) =0 ,
\end{align*}
and
\begin{align*}
& \frac{\lambda^2}{2} \partial_y (|H|^2) +\la A_{\nabla^{\perp}_{\partial x}H}(\partial x),\partial y\ra + \la A_{\nabla^{\perp}_{\partial y}H}(\partial y),\partial y\ra
+ R^{N}(\partial x,H,\partial y ,\partial x) =0 ,
\end{align*}
which is equivalent to
\begin{align}\label{eq*}
& \frac{\lambda^2}{2} \partial_x (|H|^2) +\la B(\partial x,\partial x),\nabla^{\perp}_{\partial x}H\ra + \la B(\partial x,\partial y),\nabla^{\perp}_{\partial y}H\ra
+ R^{N}(\partial y,H,\partial x,\partial y)=0 ,
\end{align}
and
\begin{align}\label{eq2*}
& \frac{\lambda^2}{2} \partial_y (|H|^2) +\la B(\partial y,\partial x),\nabla^{\perp}_{\partial x}H\ra+ \la B(\partial y,\partial y),\nabla^{\perp}_{\partial y}H\ra
+ R^{N}(\partial x,H,\partial y ,\partial x) =0 .
\end{align}
Since $\partial z = \tfrac{1}{2} (\partial x - i \partial y)$ and $\partial \bar{z} = \tfrac{1}{2} (\partial x + i \partial y)$ we see that
$$B(\partial z,\partial z) = \tfrac{1}{2} ( \lambda^2 H - B(\partial y,\partial y) - i B(\partial x,\partial y) )$$
and 
$$\la B(\partial z,\partial z),H\ra = \tfrac{1}{2} ( \lambda^2 |H|^2 - \la B(\partial y,\partial y),H\ra - i \la B(\partial x,\partial y),H\ra ) .$$
Next we compute $\partial_{\bar{z}}\la B(\partial z,\partial z),H\ra$:
\begin{align*}
&(\partial x + i \partial y)\Big( \lambda^2 |H|^2 - \la B(\partial y,\partial y),H\ra - i \la B(\partial x,\partial y),H\ra \Big)\\
&= 2\lambda \frac{\partial\lambda}{\partial x} |H|^2 +\lambda^2 \partial_x (|H|^2) - \la \nabla^{\perp}_{\partial x} B(\partial y,\partial y),H\ra 
- \la B(\partial y,\partial y), \nabla^{\perp}_{\partial x} H\ra \\
&+\la \nabla^{\perp}_{\partial y} B(\partial x,\partial y),H\ra + \la B(\partial x,\partial y), \nabla^{\perp}_{\partial y} H\ra\\
&+ i \Big\{ 2\lambda \frac{\partial\lambda}{\partial y} |H|^2 +\lambda^2 \partial_y (|H|^2)
- \la \nabla^{\perp}_{\partial y} B(\partial y,\partial y),H\ra 
- \la B(\partial y,\partial y), \nabla^{\perp}_{\partial y} H\ra \\
&-\la \nabla^{\perp}_{\partial x} B(\partial x,\partial y),H\ra - \la B(\partial x,\partial y), \nabla^{\perp}_{\partial x} H\ra \Big\} \\
&= A+ iB .
\end{align*}
With Equation~\eqref{eq*}
\begin{align*}
A&= 2\lambda \frac{\partial\lambda}{\partial x} |H|^2 +\frac{1}{2}\lambda^2 \partial_x (|H|^2) - \la \nabla^{\perp}_{\partial x} B(\partial y,\partial y),H\ra 
- \la B(\partial y,\partial y), \nabla^{\perp}_{\partial x} H\ra \\
&+\la \nabla^{\perp}_{\partial y} B(\partial x,\partial y),H\ra 
- \la B(\partial x,\partial x), \nabla^{\perp}_{\partial x} H\ra - R(\partial y ,H,\partial x,\partial y)\\
&= 2\lambda \frac{\partial\lambda}{\partial x} |H|^2 +\frac{1}{2}\lambda^2 \partial_x (|H|^2) - \la \nabla^{\perp}_{\partial x} B(\partial y,\partial y),H\ra 
+\la \nabla^{\perp}_{\partial y} B(\partial x,\partial y),H\ra \\
&- \la 2\lambda^2 H, \nabla^{\perp}_{\partial x} H\ra - R(\partial y ,H,\partial x,\partial y)\\
&= 2\lambda \frac{\partial\lambda}{\partial x} |H|^2 -\frac{1}{2}\lambda^2 \partial_x (|H|^2) - \la \nabla^{\perp}_{\partial x} B(\partial y,\partial y),H\ra 
+\la \nabla^{\perp}_{\partial y} B(\partial x,\partial y),H\ra \\
&- R(\partial y ,H,\partial x,\partial y) .
\end{align*}
From the Codazzi equation
\begin{align*}
\la \nabla^{\perp}_{\partial y} B(\partial x,&\partial y),H\ra \\
&= \la (\nabla^{\perp}_{\partial y} B)(\partial x,\partial y),H\ra 
+ \la B(\nabla_{\partial y}\partial x,\partial y),H\ra +\la B(\partial x,\nabla_{\partial y}\partial y),H\ra \\
&= \la \nabla^{\perp}_{\partial x} B(\partial y,\partial y),H\ra -2 \la B(\nabla_{\partial x} \partial y,\partial y),H\ra 
+ R(\partial y,\partial x,H,\partial y)\\
&+ \la B(\nabla_{\partial y}\partial x,\partial y),H\ra +\la B(\partial x,\nabla_{\partial y}\partial y),H\ra ,
\end{align*}
therefore
\begin{align*}
A&= 2\lambda \frac{\partial\lambda}{\partial x} |H|^2 -\frac{1}{2}\lambda^2 \partial_x (|H|^2)- \la \nabla^{\perp}_{\partial x} B(\partial y,\partial y),H\ra 
+\la \nabla^{\perp}_{\partial x} B(\partial y,\partial y),H\ra \\
&- \la B(\nabla_{\partial x} \partial y,\partial y),H\ra 
+ \la B(\partial x,\nabla_{\partial y} \partial y),H\ra + R(\partial y,\partial x,H,\partial y)- R(\partial y ,H,\partial x,\partial y) \\
&= 2\lambda \frac{\partial\lambda}{\partial x} |H|^2 -\frac{1}{2}\lambda^2 \partial_x (|H|^2)
- \la B(\tfrac{1}{\lambda}(\tfrac{\partial\lambda}{\partial y}\partial x+\tfrac{\partial\lambda}{\partial x}\partial y),\partial y),H\ra \\
&+ \la B(\tfrac{1}{\lambda}(-\tfrac{\partial\lambda}{\partial x}\partial x+\tfrac{\partial\lambda}{\partial y}\partial y),\partial x),H\ra \\
&=-\frac{1}{2}\lambda^2 \partial_x (|H|^2) .
\end{align*}
Identical arguments for the imaginary part $B$, using~\eqref{eq2*},  yield
\begin{align*}
B &=\frac{1}{2}\lambda^2 \partial_y (|H|^2).
\end{align*}
\end{proof}

\begin{remark}
If $\phi : (M^2 ,g) \to (N^n ,h)$ is a proper-biharmonic Riemannian immersion with constant mean curvature, $M^2$ oriented. Then $\la B(\partial z,\partial z),H\ra dz^2$ is globally defined and, if $M^2$ has no pseudo-umbilical point, it is equal to $\frac{1}{4} dz^2$ and therefore $M^2$ is an affine manifold.
\end{remark}

\begin{corollary}
Let $\phi : (M^2 ,g) \to (N^n ,h)$ be a proper-biharmonic Riemannian immersion with constant mean curvature, $M^2$ oriented. If $M^2$ is not pseudo-umbilical then its pseudo-umbilical points are isolated.
\end{corollary}

Combining Theorem~\ref{thm3} with \cite[Lemma 1 page 59]{Kotani}, yields
\begin{theorem}
Let $\phi : (M^2 ,g) \to (N^n ,h)$ be a proper-biharmonic Riemannian immersion with constant mean curvature $H$.
If $M^2$ is a topological sphere $\sn^2$ then $M$ is pseudo-umbilical.
\end{theorem}
\begin{proof}
Since $\la B(\partial z,\partial z),H\ra =0$, we have
\begin{align*}
&\la B(\partial x,\partial x)- B(\partial y,\partial y),H\ra =0 \mbox{ and } \, \la B(\partial x,\partial y),H\ra =0 ,
\end{align*}
which is equivalent to
\begin{align*}
&\la A_{H}(\partial x),\partial x\ra = \la A_{H}(\partial y),\partial y\ra \mbox{ and } \, \la A_{H}(\partial x),\partial y\ra = \la A_{H}(\partial y),\partial x\ra =0 .
\end{align*}
\end{proof}


\begin{thebibliography}{999}


\bibitem{BLO}
P. Baird, E. Loubeau and C. Oniciuc,
\newblock Harmonic and biharmonic maps from surfaces, 
\newblock {\em Contemporary Mathematics} {\bf 542} (2011), 223--230.

\bibitem{Fetcu}
D. Fetcu and A.~L. Pinheiro,
\newblock Biharmonic surfaces with parallel mean curvature in complex space forms, 
\newblock {\em preprint} arXiv: 1303.4279v1.

\bibitem{Hopf}
H. Hopf,
\newblock  Differential geometry in the large, 
\newblock Lecture Notes in Mathematics 1000, Springer-Verlag (1983).

\bibitem{Jiang1}
G. Y. Jiang,
\newblock 2-Harmonic maps and their first and second variational formulas, 
\newblock {\em Chin. Ann. Math.}, Ser. A {\bf 7}(4) (1986), 389--402.

\bibitem{Jiang2}
G. Y. Jiang,
\newblock The conservation law for 2-harmonic maps between Riemannian manifolds, 
\newblock {\em Acta Math. Sinica} {\bf 30} (1987), 220--225.

\bibitem{Kotani}
M. Kotani,
\newblock A decomposition theorem of $2$-type immersions, 
\newblock {\em Nagoya Math. J.} {\bf 118} (1990), 55--64.

\bibitem{LMO}
E. Loubeau, S. Montaldo and C. Oniciuc,
\newblock The stress-energy tensor for biharmonic maps,
\newblock {\em Math. Zeit.} {\bf 259} (2008), 503--524.

\bibitem{MO} 
S.~Montaldo and C.~Oniciuc,
\newblock A short survey on biharmonic maps between Riemannian manifolds,
\newblock {\em Rev. Un. Mat. Argentina} {\bf 47} (2) (2006), 1--22.

\bibitem{OZ}
Y.-L.~Ou and Z.-P.~Wang,
\newblock Constant mean curvature and totally umbilical biharmonic surfaces in $3$-dimensional geometries,
\newblock {\em J. Geom. Phys.} {\bf 61} (2011), 1845--1853.

\bibitem{S} 
T.~Sasahara, 
\newblock Biharmonic Lagrangian surfaces of constant mean curvature in complex space forms, 
\newblock {\em Glasg. Math. J.} {\bf 49} (2007), 497--507.


\end{thebibliography}
\end{document}